\documentclass[11pt]{article} 
\usepackage{amsmath,amssymb,amsthm}  

\allowdisplaybreaks

\begin{document}

\def\fl#1{\left\lfloor#1\right\rfloor}
\def\cl#1{\left\lceil#1\right\rceil}
\def\ang#1{\left\langle#1\right\rangle}
\def\stf#1#2{\left[#1\atop#2\right]} 
\def\sts#1#2{\left\{#1\atop#2\right\}}
\def\eul#1#2{\left\langle#1\atop#2\right\rangle}
\def\N{\mathbb N}
\def\Z{\mathbb Z}
\def\R{\mathbb R}
\def\C{\mathbb C}

\newtheorem{theorem}{Theorem}
\newtheorem{Prop}{Proposition}
\newtheorem{Cor}{Corollary}
\newtheorem{Lem}{Lemma}

\newenvironment{Rem}{\begin{trivlist} \item[\hskip \labelsep{\it
Remark.}]\setlength{\parindent}{0pt}}{\end{trivlist}}

\title{On the determination of $p$-Frobenius and related numbers using the $p$-Ap\'ery set 
}

\author{
Takao Komatsu 
\\
\small Department of Mathematical Sciences, School of Science\\[-0.8ex]
\small Zhejiang Sci-Tech University\\[-0.8ex]
\small Hangzhou 310018 China\\[-0.8ex]
\small \texttt{komatsu@zstu.edu.cn}
}

\date{
%\small Submitted: February 15, 2021;  Accepted: March 25, 2021.\\
\small MR Subject Classifications: Primary 11D07; Secondary 05A15, 05A17, 05A19, 11B68, 11D04, 11P81, 20M14 
}

\maketitle
 
\begin{abstract} 
In this paper, we give convenient formulas in order to obtain explicit expressions of a generalized Frobenius number called the $p$-Frobenius number as well as its related values. 
Here, for a non-negative integer $p$, the $p$-Frobenius number is the largest integer whose number of solutions of the linear diophantine equation in terms of positive integers $a_1,a_2,\dots,a_k$ with $\gcd(a_1,a_2,\dots,a_k)=1$ is at most $p$.  
When $p=0$, the problem is reduced to the famous and classical linear Diophantine problem of Frobenius. $0$-Frobenius number is the classical Frobenius number.  
Our formula is not only a natural extension of the existing classical formulas, but also has the great advantage that the explicit expressions of values such as the $p$-Frobenius and related numbers can be obtained systematically.   
The concept and formula of the weighted sum has been given recently. We also give a $p$-generalized formula for such weighted sums. The central role is the $p$-Ap\'ery set, which is a generalization of the classical Ap\'ery set.   
\end{abstract}

\section{Introduction}  

The coin exchange problem is a problem of finding the maximum amount that cannot be paid exactly with the specified coins. It has a long history and is one of the problems that has attracted many people as well as experts. It is also known as the postage stamp problem or the chicken McNugget problem.  
Given positive integers $a_1,\dots,a_k$ with $\gcd(a_1,\dots,a_k)=1$, it is well-known that all sufficiently large $n$ can be represented as a non-negative integer combination of $a_1,\dots,a_k$. Nowadays, it is most known as 
the {\it linear Diophantine problem of Frobenius}, which is to determine the largest positive integer that is not representable as a non-negative integer combination of given positive integers that are coprime (see \cite{ra05} for general references). This number is denoted by $g(a_1,\dots,a_k)$ and is often called the {\it Frobenius number}. 
Along with the Frobenius number, various quantities have been raised and studied.  
The cardinality and the sum of non-negative integers with no non-negative integer representation by $a_1,\dots,a_k$ are denoted by $n(a_1,\dots,a_k)$ and $s(a_1,\dots,a_k)$, respectively. The former is called the {\it genus} in numerical semigroup (see, e.g., \cite{ADG20}) or the {\it Sylvester number}, and the latter the {\it Sylvester sum}.   

In this paper, we introduce a generalization of the Frobenius problem associated with the number of representations (solutions). 
For a non-negative integer $n$, 
let the {\it denumerant} $d(n)=d(n;a_1,\dots,a_k)$ be the number of non-negative integer solutions, that is, 
$$
d(n;a_1,\dots.a_k):=\#\{(x_1,\dots,x_k)|n=x_1 a_1+\cdots+x_k a_k,\,x_i\ge 0\,(1\le i\le k)\}\,.
$$ 
This is actually the number of partitions of $n$ whose summands are taken (repetitions allowed) from the sequence $a_1,\dots,a_k$. The generating function of $d(n;a_1,\dots,a_k)$ is given by 
$$
\sum_{n=0}^\infty d(n|a_1,\dots,a_k)z^n=\frac{1}{(1-z^{a_1})\cdots(1-z^{a_k})}
$$ 
(\cite{sy1882}).   
Sylvester \cite{sy1857} and Cayley \cite{Cayley} showed that $d(n;a_1,a_2,\dots,a_k)$ can be expressed as the sum of a polynomial in $n$ of degree $k-1$ and a periodic function of period $a_1 a_2\cdots a_k$.   
When $k=2$, the formula for $d(n|a,b)$ is given by Tripathi \cite{tr00}. When $k=3$, there is a recent work by Binner (\cite{bi20}).   
For a non-negative integer $p$, define $S_p:=S_p(a_1,a_2,\dots,a_k)$ by 
$$
S_p(a_1,a_2,\dots,a_k)=\{n\in\mathbb N_0|d(n;a_1,a_2,\dots,a_k)>p\}\,, 
$$ 
where $\mathbb N_0$ is the set of non-negative integers.  

Define $g_p(a_1,a_2,\dots,a_k)$, $n_p(a_1,a_2,\dots,a_k)$ and $s_p(a_1,a_2,\dots,a_k)$ by 
\begin{align*}
g_p(a_1,a_2,\dots,a_k)&=\max_{n\in \mathbb N_0\backslash S_p}n=\max_{d(n)\le p}n,\\ 
n_p(a_1,a_2,\dots,a_k)&=\sum_{n\in \mathbb N_0\backslash S_p}1=\sum_{d(n)\le p}1,\\ 
s_p(a_1,a_2,\dots,a_k)&=\sum_{n\in \mathbb N_0\backslash S_p}n=\sum_{d(n)\le p}n\,,
\end{align*}
respectively, and they are called the {\it $p$-Frobenius number}, the {\it $p$-Sylvester number} (or {\it $p$-genus}) and the {\it $p$-Sylvester sum}, respectively. 
When $p=0$, 
\begin{align*}
&g(a_1,a_2,\dots,a_k)=g_0(a_1,a_2,\dots,a_k),\quad n(a_1,a_2,\dots,a_k)=n_0(a_1,a_2,\dots,a_k)\\
&\text{and}\quad s(a_1,a_2,\dots,a_k)=s_0(a_1,a_2,\dots,a_k)
\end{align*} 
are the original Frobenius number, Sylvester number (or genus) and Sylvester sum, respectively. $S=S_0$ is the original numerical semigroup. As $S_p\cup\{0\}$ itself is a numerical semigroup and its maximal ideal is $S_p$, we call $S_p$ the {\it $p$-numerical semigroup} (\cite{KY24}).

In this paper, with the help of a generalized Ap\'ery set called the $p$-Ap\'ery set, we give a general formula for power sums    
\begin{equation}
s_p^{(\mu)}(a_1,a_2,\dots,a_k):=\sum_{d(n)\le p}n^\mu\,,
\label{eq:powersum}
\end{equation} 
where $\mu$ is a non-negative integer, 
by using Bernoulli numbers $B_n$, defined by the generating function 
\begin{equation}
\frac{x}{e^x-1}=\sum_{n=0}^\infty B_n\frac{x^n}{n!}\,. 
\label{eq:ber}
\end{equation}  
The first few values of Bernoulli numbers are 
$$
\{B_n\}_{n\ge 0}=1,-\frac{1}{2},\frac{1}{6},0,-\frac{1}{30},0,\frac{1}{42},0,-\frac{1}{30},0,\frac{5}{66},0,-\frac{691}{2730},0,\frac{7}{6},\dots\,. 
$$ 
When $p=0$ (the classical case), the corresponding result can be found in \cite[Theorem 1]{ko22}. 
By taking $\mu=0,1$ in (\ref{eq:powersum}), it is reduced to the formulas of the Sylvester number and sum, respectively.  We also give a general formula for weighted sums 
$$
s^{(\mu)}_{\lambda,p}(a_1,\dots,a_k):=\sum_{d(n)\le p}\lambda^n n^\mu\,, 
$$ 
where $\lambda\ne 1$ and $\mu$ is a non-negative integer. Slightly simpler forms for weighted sums are studied in \cite{ko22,KZ0,KZ}.  

When $p=0$ and $k=2$, it is not difficult in finding explicit forms for the Frobenius and Sylvester numbers.
According to Sylvester, for positive integers $a$ and $b$ with $\gcd(a,b)=1$,  
\begin{align*}
g(a,b)&=(a-1)(b-1)-1\quad{\rm \cite{sy1884}}\,,\\
n(a,b)&=\frac{1}{2}(a-1)(b-1)\quad{\rm \cite{sy1882}}\,. 
\end{align*}
The explicit formula of the Sylvester sum was discovered by Brown and Shiue \cite{bs93} much later than the discovery of the explicit formulas of these two numbers. 
\begin{equation}
s(a,b)=\frac{1}{12}(a-1)(b-1)(2 a b-a-b-1)\,.  
\label{brown}
\end{equation}
R\o dseth \cite{ro94} generalized the result of Brown and Shiue to give a closed formula for $s^{(\mu)}(a,b)$ by using Bernoulli numbers. 
Recently, the weighted (power) sum $s^{(\mu)}_{\lambda,0}(a_1,\dots,a_k)$ is also studied (\cite{KZ0,KZ}) for a non-negative integer $\mu$ and $\lambda\ne 1$.   
    
When $k\ge 3$, exact determination of the Frobenius numbers is extremely difficult.  
The Frobenius number cannot be given by closed formulas of a certain type (\cite{cu90}), the problem to determine $g(a_1,\dots,a_k)$ is NP-hard under Turing reduction (see, e.g., Ram\'irez Alfons\'in \cite{ra05}). 
One of the convenient algorithms is given by Johnson \cite{jo60}. One analytic approach to the Frobenius number can be seen in \cite{bgk01,ko03}. Only for several special cases, explicit formula of the Frobenius numbers can be given (see, e.g., \cite{RR18} and reference therein).

When $p\ge 1$, explicit formulas can be given for $k=2$. But, it is even harder to find an explicit formula for $k\ge 3$. In fact, up to the present, no concrete example of explicit form had been found. However, just recently, by using the results established in this paper, we have been finally successful to find explicit formulas for some special cases (see, e.g., \cite{Ko23a,Ko23b,KM,KP,KY}). 

Such a $p$-generalization or similar ones have already been proposed and studied by previous researchers. There are some results about bounds or algorithms for generalized Frobenius numbers (see, e.g., \cite{ADL16,BDFHKMRSS,FS11}). For example, some consider $S_p^\ast$ as the set of integers whose nonnegative integral linear combinations of given positive integers $a_1,a_2,\dots,a_k$ are expressed in {\it exactly} $p$ ways (see, e.g., \cite{BDFHKMRSS,FS11}). Then, 
the corresponding $p$-Frobenius number $g_p^\ast(a_1,a_2,\dots,a_k)$ is the largest integer that has {\it\underline{exactly $p$ distinct}} representations. 
The $p$-genus may also be defined in a different way. However, there is no convenient formula to obtain such values. 
In our $p$-generalization, there is a big advantage. Once we know the detailed structure of the $p$-Ap\'ery set, by using the convenient formula shown in this paper, we are able to obtain our $p$-Frobenius, $p$-Sylvester numbers, and so on without difficulty.

\section{Main results}   

Without loss of generality, we assume that $a_1=\min_{1\le i\le k}a_i$.  
In order to obtain the main result in this paper, we introduce the {\it $p$-Ap\'ery set} of $A=\{a_1,a_2,\dots,a_k\}$ for a non-negative integer $p$:  
$$
%{\rm Ap}_p(A)=
{\rm Ap}_p(a_1,a_2,\dots,a_k)=\left\{m_0^{(p)},m_1^{(p)},\dots,m_{a_1-1}^{(p)}\right\}\,, 
$$ 
where for $0\le i\le a_1-1$, each non-negative integer $m_i^{(p)}$ satisfies the following three conditions:  
\begin{enumerate}
\item[(1)] $m_i^{(p)}\equiv i\pmod{a_1}$
\item[(2)] $m_i^{(p)}\in S_p(a_1,a_2,\dots,a_k)$
\item[(3)] $m_i^{(p)}-a_1\not\in S_p(a_1,a_2,\dots,a_k)$
\end{enumerate}
Therefore, each $m_i^{(p)}$ ($0\le i\le a_1-1$) is determined uniquely.  
When $p=0$, $0$-Ap\'ery set is the classical Ap\'ery set (see \cite{Apery}).  Hence, 
the $p$-Ap\'ery set ${\rm Ap}_p(A)$ is congruent to the set 
$$
\{0,1,\dots,a_1-1\}\pmod{a_1}\,.   
$$  
Note that $m_0^{(0)}=0$. 

By using the elements of the $p$-Ap\'ery set, our main theorem is stated as follows. This is a $p$-generalization of \cite[Theorem 1]{ko22}.  

\begin{theorem} 
Let $k$, $p$ and $\mu$ be integers with $k\ge 2$, $p\ge 0$ and $\mu\ge 0$.  
Assume that $\gcd(a_1,a_2,\dots,a_k)=1$.  We have 
\begin{align*} 
&s_p^{(\mu)}(a_1,a_2,\dots,a_k):=\sum_{d(n)\le p}n^\mu\\ 
&=\frac{1}{\mu+1}\sum_{\kappa=0}^{\mu}\binom{\mu+1}{\kappa}B_{\kappa}a_1^{\kappa-1}\sum_{i=0}^{a_1-1}\bigl(m_i^{(p)}\bigr)^{\mu+1-\kappa} 
+\frac{B_{\mu+1}}{\mu+1}(a_1^{\mu+1}-1)\,, 
\end{align*} 
where $B_n$ are Bernoulli numbers, defined in (\ref{eq:ber}), and $\{m_0^{(p)},m_1^{(p)},\dots,m_{a_1-1}^{(p)}\}$ is the $p$-ap\'ery set of $A=\{a_1,a_2,\dots,a_k\}$ and $a_1=\min_{1\le i\le k}a_i$.   
\label{th-mp}
\end{theorem} 

The proof of Theorem \ref{th-mp} shall be done in the following subsection.  
\bigskip

From Theorem \ref{th-mp}, we can give formulas for the $p$-Frobenius number, the $p$-Sylvester number ($p$-genus) and the $p$-Sylvester sum, respectively.  

\begin{Cor}  
Let $k$, $p$ and $\mu$ be integers with $k\ge 2$ and $p\ge 0$.  
Assume that $\gcd(a_1,a_2,\dots,a_k)=1$.  We have 
\begin{align}  
g_p(a_1,a_2,\dots,a_k)&=\max_{0\le i\le a_1-1}m_i^{(p)}-a_1
\label{mp-g}\,,\\  
n_p(a_1,a_2,\dots,a_k)&=\frac{1}{a_1}\sum_{i=0}^{a_1-1}m_i^{(p)}-\frac{a_1-1}{2}\,,
\label{mp-n}\\
s_p(a_1,a_2,\dots,a_k)&=\frac{1}{2 a_1}\sum_{i=0}^{a_1-1}\bigl(m_i^{(p)}\bigr)^2-\frac{1}{2}\sum_{i=0}^{a_1-1}m_i^{(p)}+\frac{a_1^2-1}{12}\,.
\label{mp-s}
\end{align}
\label{cor-mp}
\end{Cor}

\noindent 
{\it Remark.}  
When $p=0$ in Corollary \ref{cor-mp}, the formulas are reduced to the classically known formulas: 
\begin{align*}
g(a_1,a_2,\dots,a_k)&=\left(\max_{1\le i\le a_1-1}m_i\right)-a_1\,,\quad{\rm \cite{bs62}}\\ 
n(a_1,a_2,\dots,a_k)&=\frac{1}{a_1}\sum_{i=1}^{a_1-1}m_i-\frac{a_1-1}{2}\,,\quad{\rm \cite{se77}}\\ 
s(a_1,a_2,\dots,a_k)&=\frac{1}{2 a_1}\sum_{i=1}^{a_1-1}m_i^2-\frac{1}{2}\sum_{i=1}^{a_1-1}m_i+\frac{a_1^2-1}{12}\,,\quad{\rm \cite{tr08}}
\end{align*}
where $m_i=m_i^{(0)}$ ($1\le i\le a_1-1$). 
Since $m_0=m_0^{(0)}=0$ is applied in the classical formulas, the sum runs from $i=1$. 

\begin{proof}[Proof of Corollary \ref{cor-mp}.] 
(\ref{mp-g}) 
It is trivial that if $r=\max_{0\le i\le a_1-1}m_i^{(p)}$ is the largest integer whose number of representations is $p+1$ or larger, then $r-a_1$ is the largest positive integer whose number of representations is $p$ or smaller. 

The formulas (\ref{mp-n}) and (\ref{mp-s}) are reduced from Theorem \ref{th-mp} for $\mu=0$ and $\mu=1$, respectively.  
\end{proof}

In the case of two variables, namely, $a_1=a$ and $a_2=b$, by 
\begin{equation} 
\{m_i^{(p)}|0\le i\le a-1\}=\{b(p a+i)|0\le i\le a-1\}\,,
\label{eq:p=2}
\end{equation} 
Corollary \ref{cor-mp} is reduced to the formulas for $k=2$.  

\begin{Cor}[\cite{bb20}] 
For a non-negative integer $p$, we have 
\begin{align*}
g_p(a,b)&=(p+1)a b-a-b\,,\\
n_p(a,b)&=\frac{1}{2}\bigl((2 p+1)a b-a-b+1\bigr)\,,\\
s_p(a,b)&=\frac{1}{12}\bigl(2(3 p^2+3 p+1)a^2 b^2-3(2 p+1)a b(a+b)\notag\\
&\qquad +a^2+b^2+3 a b-1
\bigr)\,.
\end{align*}
\label{cor-mp2} 
\end{Cor}

\subsection{Proof of Theorem \ref{th-mp}} 

For convenience, put $a=a_1$. From the definition of $m_i^{(p)}$, for each $i$, there exists a non-negative integer $\ell_i$ such that 
$$
m_i^{(p)}-a_1, m_i^{(p)}-2 a_1,\dots, m_i^{(p)}-\ell_i a_1\in\mathbb N_0\backslash S_p(a_1,a_2,\dots,a_k)
$$ 
with $m_i^{(p)}-\ell_i a_1>0$ and $m_i^{(p)}-(\ell_i+1)a_1<0$.  Since $m_i^{(p)}\equiv i\pmod{a_1}$, we see $\ell_i=(m_i^{(p)}-i)/a_1$ for $1\le i\le a_1-1$, $\ell_0=(m_0^{(p)}-a_1)/a_1$ for $i=0$.   
Then we have 
\begin{equation}
s_p^{(\mu)}(a_1,a_2,\dots,a_k) 
=\sum_{i=1}^{a-1}\sum_{j=1}^{\ell_i}(m_i^{(p)}-j a)^\mu+\sum_{j=1}^{\ell_i}(m_i^{(p)}-j a)^\mu\,.\label{eq:smu} 
\end{equation}  
First, consider the first part for $1\le i\le a-1$ in (\ref{eq:smu}).  
Since 
$$
\sum_{j=1}^{\ell_i} j^n=\sum_{\kappa=0}^n\binom{n}{\kappa}(-1)^\kappa B_\kappa\frac{{\ell_i}^{n+1-\kappa}}{n+1-\kappa}\,, 
$$  
\begin{align}
&\sum_{i=1}^{a-1}\sum_{j=1}^{\ell_i}(m_i^{(p)}-j a)^\mu\notag\\ 
&=\sum_{i=1}^{a-1}\sum_{j=1}^{\ell_i}\sum_{n=0}^\mu\binom{\mu}{n}\bigl(m_i^{(p)}\bigr)^{\mu-n}j^n(-a)^n\notag\\ 
&=\sum_{i=1}^{a-1}\sum_{n=0}^\mu\binom{\mu}{n}\bigl(m_i^{(p)}\bigr)^{\mu-n}(-a)^n\sum_{\kappa=0}^n\binom{n}{\kappa}(-1)^\kappa B_\kappa\frac{1}{n+1-\kappa}\left(\frac{m_i^{(p)}-i}{a}\right)^{n+1-\kappa}\notag\\
&=\sum_{n=0}^\mu\binom{\mu}{n}\sum_{\kappa=0}^n\binom{n}{\kappa}\frac{(-1)^{n-\kappa}a^{\kappa-1}B_\kappa}{n+1-\kappa}\sum_{i=1}^{a-1}\sum_{j=0}^{n+1-\kappa}\binom{n+1-\kappa}{j}\bigl(m_i^{(p)}\bigr)^{\mu+1-\kappa-j}(-i)^j\,. 
\label{eq:234}
\end{align}
%%%%%%%%%%%%%%%%%%%%%%%%%%%%%%%%%%%%%%%
When $j=0$ in (\ref{eq:234}), we can obtain 
\begin{align*}  
&\sum_{n=0}^\mu\binom{\mu}{n}\sum_{\kappa=0}^n\binom{n}{\kappa}\frac{(-1)^{n-\kappa}a^{\kappa-1}B_\kappa}{n+1-\kappa}\sum_{i=1}^{a-1}\bigl(m_i^{(p)}\bigr)^{\mu+1-\kappa}\\ 
&=\sum_{\kappa=0}^\mu\sum_{n=\kappa}^\mu\binom{\mu}{n}\binom{n}{\kappa}\frac{(-1)^{n-\kappa}a^{\kappa-1}B_\kappa}{n+1-\kappa}\sum_{i=1}^{a-1}\bigl(m_i^{(p)}\bigr)^{\mu+1-\kappa}\\ 
&=\frac{1}{\mu+1}\sum_{\kappa=0}^{\mu}\binom{\mu+1}{\kappa}B_{\kappa}a^{\kappa-1}\sum_{i=1}^{a-1}\bigl(m_i^{(p)}\bigr)^{\mu+1-\kappa}  
\end{align*} 
because 
\begin{align*}  
\sum_{n=\kappa}^\mu\binom{\mu}{n}\binom{n}{\kappa}\frac{(-1)^{n-\kappa}}{n+1-\kappa} 
&=\frac{1}{\mu+1}\binom{\mu+1}{\kappa}\sum_{n=\kappa}^\mu\binom{\mu-\kappa+1}{n+1-\kappa}(-1)^{n-\kappa}\\
&=\frac{1}{\mu+1}\binom{\mu+1}{\kappa}\left(1+\sum_{j=0}^{\mu-\kappa+1}\binom{\mu-\kappa+1}{j}(-1)^{j-1}\right)\\ 
&=\frac{1}{\mu+1}\binom{\mu+1}{\kappa}\,.  
\end{align*} 
%%%%%%%%%%%%%%%%%%%%%%%%%%%%%%%%%

When $0<j\le n+1-\kappa<\mu+1-\kappa$ in (\ref{eq:234}), by 
\begin{align*}
&\sum_{n=\kappa}^\mu\binom{\mu}{n}\binom{n}{\kappa}\frac{(-1)^{n-\kappa}}{n+1-\kappa}\binom{n+1-\kappa}{j}\\
&=\frac{(-1)^{j-1}\mu!}{j!\kappa!(\mu+1-\kappa-j)!}\sum_{h=0}^{\mu-\kappa}\binom{\mu+1-\kappa-j}{h}(-1)^{\mu+1-\kappa-j-h}\\
&=0\,, 
\end{align*}
all the terms of $\bigl(m_i^{(p)}\bigr)^{\mu+1-\kappa-j}(-i)^j$ for $\mu+1-\kappa-j>0$ and $j>0$ vanished. 

%%%%%%%%%%%%%%%%%%%%%%%%%%%%%%%%%%%
When $n=\mu$ and $j=\mu+1-\kappa$ in (\ref{eq:234}), we have 
\begin{align}  
&\sum_{\kappa=0}^\mu\binom{\mu}{\kappa}\frac{(-1)^{\mu-\kappa}a^{\kappa-1}B_\kappa}{\mu+1-\kappa}\sum_{i=1}^{a-1}(-i)^{\mu+1-\kappa}\notag\\
&=-\sum_{\kappa=0}^\mu\binom{\mu}{\kappa}\frac{a^{\kappa-1}B_\kappa}{\mu+1-\kappa}\sum_{i=1}^{a-1}i^{\mu+1-\kappa}\notag\\
&=-\sum_{\kappa=0}^\mu\binom{\mu}{\kappa}\frac{a^{\kappa-1}B_\kappa}{\mu+1-\kappa}\sum_{k=0}^{\mu+1-\kappa}\binom{\mu+1-\kappa}{k}B_{\mu+1-\kappa-k}\frac{a^{k+1}}{k+1}\notag\\
&=-\sum_{\kappa=0}^\mu\binom{\mu}{\kappa}\frac{B_\kappa B_{\mu+1-\kappa}}{\mu+1-\kappa}a^\kappa\notag\\
&\quad -\sum_{\kappa=0}^\mu\sum_{k=1}^{\mu+1-\kappa}\binom{\mu}{\kappa}\binom{\mu+1-\kappa}{k}\frac{B_\kappa B_{\mu+1-\kappa-k}}{(\mu+1-\kappa)(k+1)}a^{\kappa+k}\notag\\
&=-\sum_{l=0}^\mu\binom{\mu}{l}\frac{B_l B_{\mu+1-l}}{\mu+1-l}a^l\notag\\
&\quad -\sum_{l=1}^{\mu+1}\sum_{\kappa=0}^{l-1}\binom{\mu}{\kappa}\binom{\mu+1-\kappa}{l-\kappa}\frac{B_\kappa B_{\mu+1-l}}{(\mu+1-\kappa)(l+1-\kappa)}a^{l}\,. 
\label{eq:20}
\end{align}
By using the recurrence relation 
$$
\sum_{\kappa=0}^{\mu+1}\binom{\mu+2}{\kappa}B_\kappa=0\,, 
$$ 
the term of $a^{\mu+1}$ is yielded from the second sum in (\ref{eq:20}) as 
$$
-\sum_{\kappa=0}^\mu\binom{\mu}{\kappa}\frac{B_\kappa a^{\mu+1}}{(\mu+1-\kappa)(\mu+2-\kappa)}=\frac{B_{\mu+1}}{\mu+1}a^{\mu+1}\,.  
$$   
The constant term is yielded from the first sum in (\ref{eq:20}) as 
$$
-\frac{B_{\mu+1}}{\mu+1}\,. 
$$ 
Other terms of $a^l$ ($1\le l\le\mu$) are canceled, so vanished, because by 
$$
\sum_{\kappa=0}^l\binom{l+1}{\kappa}B_\kappa=0\,, 
$$ 
we get 
$$
\sum_{\kappa=0}^{l-1}\binom{\mu}{\kappa}\binom{\mu+1-\kappa}{l-\kappa}\frac{B_\kappa}{(\mu+1-\kappa)(l+1-\kappa)}=-\binom{\mu}{l}\frac{B_l}{\mu+1-l}\,.  
$$ 
Hence, we obtain the term 
\begin{align*} 
\sum_{\kappa=0}^{\mu}\binom{\mu}{\kappa}\frac{(-1)^{\mu-\kappa}a^{\kappa-1}B_\kappa}{\mu+1-\kappa}\sum_{i=1}^{a-1}(-i)^{\mu+1-\kappa}
=\frac{B_{\mu+1}}{\mu+1}(a^{\mu+1}-1)\,. 
\end{align*} 

Next, consider the second term for $i=0$ in (\ref{eq:smu}). In a similar manner, for $j=0$ we have 
$$ 
\frac{1}{\mu+1}\sum_{\kappa=0}^{\mu}\binom{\mu+1}{\kappa}B_{\kappa}a^{\kappa-1}\bigl(m_0^{(p)}\bigr)^{\mu+1-\kappa}\,. 
$$ 
For $0<j<\mu+1-\kappa$, all the terms are vanished.  
For $n=\mu$ and $j=\mu+1-\kappa$, we have 
\begin{align*}  
&\sum_{\kappa=0}^\mu\binom{\mu}{\kappa}\frac{(-1)^{\mu-\kappa}a^{\kappa-1}B_\kappa}{\mu+1-\kappa}(-a)^{\mu+1-\kappa}\\
&=-\frac{a^\mu}{\mu+1}\sum_{\kappa=0}^\mu\binom{\mu+1}{\kappa}B_\kappa\\
&=0\,. 
\end{align*}
Combining all the terms together, we get the desired formula.

\subsection{Examples}  

Consider the sequence $5,7,11$. The numbers of representations of each $n$ are given by Table 1.  

\begin{table}[h] 
\begin{center} 
\begin{tabular}{|c|r@{~}r@{~}r@{~}r@{~}r@{~}r@{~}r@{~}r@{~}r@{~}r@{~}r@{~}r@{~}r@{~}r@{~}r@{~}r@{~}r@{~}r@{~}r@{~}r|}\hline
$n$&1&2&3&4&5&6&7&8&9&10&11&12&13&14&15&16&17&18&19&20\\ \hline 
$d(n)$&0&0&0&0&&0&&0&0&&&&0&&&&&&&\\ 
&&&&&1&&1&&&1&1&1&&1&1&1&1&1&1&1\\ \hline
$n$&21&22&23&24&25&26&27&28&29&30&31&32&33&34&35&36&37&38&39&40\\ \hline  
$d(n)$&2&2&&&2&2&2&2&2&2&2&&&2&&&&&&4\\ 
&&&1&1&&&&&&&&3&3&&3&3&3&3&3&\\ \hline
$n$&41&42&43&44&45&46&47&48&49&50&51&52&53&54&55&56&57&58&59&60\\ \hline 
$d(n)$&&4&4&4&4&4&&4&&&&&&6&6&6&6&6&6&\\ 
&3&&&&&&5&&5&5&5&5&5&&&&&&&7\\ \hline
$n$&61&62&63&64&65&66&67&68&69&70&71&72&73&74&75&76&77&78&79&80\\ \hline  
$d(n)$&7&7&7&7&&&&&&9&9&9&9&9&&&11&&&11\\ 
&&&&&8&8&8&8&8&&&&&&10&10&&10&10&\\ \hline
$n$&81&82&83&84&85&86&87&88&89&90&91&92&93&94&95&96&97&98&99&100\\ \hline 
$d(n)$&11&&11&&&&13&13&13&13&&&&&15&15&15&&&\\ 
&&12&&12&12&12&&&&&14&14&14&14&&&&16&16&16\\\hline
\end{tabular}\\ 
{Table 1.}  
\end{center}
\end{table}

For example, let $\mu=4$. Since $m_0^{(4)}=50$, $m_1^{(4)}=51$, $m_2^{(4)}=47$, $m_3^{(4)}=53$ and $m_4^{(4)}=49$, by Corollary \ref{cor-mp} we have  
\begin{align*}
&g^{(4)}(5,7,11)=53-5=48\,,\\ 
&n^{(4)}(5,7,11)=\frac{50+51+47+53+49}{5}-\frac{5-1}{2}=48\,,\\ 
&s^{(4)}(5,7,11)&\\
&=\frac{50^2+51^2+47^2+53^2+49^2}{2\cdot 5}-\frac{50+51+47+53+49}{2}+\frac{5^2-1}{12}\\
&=1129\,.
\end{align*}  
It is clear that the values of $g^{(4)}$ and $n^{(4)}$ are valid from Table 1. Since $1+2+\cdots+46+48=1129$, the value of $s^{(4)}$ is valid.  

When $\mu=6$ in Theorem \ref{th-mp}, we have 
\begin{align*} 
&s_p^{(6)}(a_1,\dots,a_k)\\
&=\frac{1}{7 a_1}\sum_{i=1}^{a_1-1}\bigl(m_i^{(p)}\bigr)^7-\frac{1}{2}\sum_{i=1}^{a_1-1}\bigl(m_i^{(p)}\bigr)^6+\frac{a_1}{2}\sum_{i=1}^{a_1-1}\bigl(m_i^{(p)}\bigr)^5-\frac{a_1^3}{6}\sum_{i=1}^{a_1-1}\bigl(m_i^{(p)}\bigr)^3\\
&\quad +\frac{a_1^5}{42}\sum_{i=1}^{a_1-1}m_i^{(p)}\,. 
\end{align*}  
So, for $(5,7,11)$ and $p=4$, we have 
\begin{align*} 
&s_4^{(6)}(5,7,11)\\
&=\frac{50^7+51^7+47^7+53^7+49^7}{7\cdot 5}-\frac{50^6+51^6+47^6+53^6+49^6}{2}\\
&\quad +\frac{5(50^5+51^5+47^5+53^5+49^5)}{2}-\frac{5^3(50^3+51^3+47^3+53^3+49^3)}{6}\\
&\quad +\frac{5^5(50+51+47+53+49)}{42}\\
&=79330369495\,.
\end{align*} 
Indeed,  
$$
1^6+2^6+\cdots+46^6+48^6=79330369495\,. 
$$

\section{Weighted sums} 

In this section, for all non-negative integers whose number of solutions is less than or equal to $p$, consider the following weighted sum with weight $\lambda$:  
$$
s^{(\mu)}_{\lambda,p}(a_1,\dots,a_k):=\sum_{d(n)\le p}\lambda^n n^\mu\,, 
$$ 
where $\lambda\ne 1$ and $\mu$ is a non-negative integer.   

For this purpose, we need  
Eulerian numbers $\eul{n}{m}$, appearing in the generating function 
\begin{equation}
\sum_{k=0}^\infty k^n x^k=\frac{1}{(1-x)^{n+1}}\sum_{m=0}^{n-1}\eul{n}{m}x^{m+1}\quad(n\ge 1)
\label{eu:gf}
\end{equation} 
with $0^0=1$ and $\eul{0}{0}=1$ (\cite[p.244]{com74}), 
have an explicit formula 
$$
\eul{n}{m}=\sum_{k=0}^{m}(-1)^k\binom{n+1}{k}(m-k+1)^n
$$   
(\cite[p.243]{com74},\cite{gkp89}).  Similarly, $m_i^{(p)}$'s are elements in the $p$-Ap\'ery set.     

\begin{theorem}  
Assume that  
$\lambda^{a_1}\ne 1$. Then, for a non-negative integer $\mu$,  
\begin{align*}  
&s^{(\mu)}_{\lambda,p}(a_1,\dots,a_k)\\
&=\sum_{n=0}^\mu\frac{(-a_1)^n}{(\lambda^{a_1}-1)^{n+1}}\binom{\mu}{n}\sum_{j=0}^n\eul{n}{n-j}\lambda^{j a_1}\sum_{i=0}^{a_1-1}\bigl(m_i^{(p)}\bigr)^{\mu-n}\lambda^{m_i^{(p)}}\\
&\quad +\frac{(-1)^{\mu+1}}{(\lambda-1)^{\mu+1}}\sum_{j=0}^\mu\eul{\mu}{\mu-j}\lambda^j\,.
\end{align*}
\label{th-hh}
\end{theorem} 

\noindent 
{\it Remark.}  
If one wants to avoid $0^0=1$ (e.g., in computational calculations), we use the formula  
\begin{align*}  
&s^{(\mu)}_{\lambda,p}(a_1,\dots,a_k)\\
&=\sum_{n=0}^{\mu-1}\frac{(-a_1)^n}{(\lambda^{a_1}-1)^{n+1}}\binom{\mu}{n}\sum_{j=0}^n\eul{n}{n-j}\lambda^{j a_1}\sum_{i=0}^{a_1-1}\bigl(m_i^{(p)}\bigr)^{\mu-n}\lambda^{m_i^{(p)}}\\
&\quad +\frac{(-a_1)^\mu}{(\lambda^{a_1}-1)^{\mu+1}}\sum_{j=0}^\mu\eul{\mu}{\mu-j}\lambda^{j a_1}\sum_{i=0}^{a_1-1}\lambda^{m_i^{(p)}}\\ 
&\quad +\frac{(-1)^{\mu+1}}{(\lambda-1)^{\mu+1}}\sum_{j=0}^\mu\eul{\mu}{\mu-j}\lambda^j\,.
\end{align*}

The proof of Theorem \ref{th-hh} shall be done in the next subsection.  
\bigskip 

%%%%%%%%%%%%%%%%%%

When $\mu=1$ in Theorem \ref{th-hh}, we have the following.   

\begin{theorem} 
If %$\lambda\ne 0$ and 
$\lambda^{a_1}\ne1$, then  
\begin{align*}  
&s_{\lambda,p}(a_1,a_2,\dots,a_k):=s_{\lambda,p}^{(1)}(a_1,a_2,\dots,a_k)\\
&=\frac{1}{\lambda^{a_1}-1}\sum_{i=0}^{a_1-1}m_i^{(p)}\lambda^{m_i^{(p)}}
-\frac{a_1\lambda^{a_1}}{(\lambda^{a_1}-1)^2}\sum_{i=0}^{a_1-1}\lambda^{m_i^{(p)}}+\frac{\lambda}{(\lambda-1)^2}\,.
\end{align*} 
\label{th1}
\end{theorem}

When $k=2$, by (\ref{eq:p=2}), Theorem \ref{th1} is reduced to a formula for the weighted sum of positive integers whose number of representations of two coprime positive integers $(a,b)$ are less than or equal to $p$. 

\begin{Cor}  
Assume that %$\lambda\ne 0$, 
$\lambda^a\ne 1$ and $\lambda^b\ne 1$.   
For $p\ge 0$, we have 
\begin{align*} 
s_{\lambda,p}(a,b):&=\sum_{d(n;a,b)\le p}\lambda^{n}n\\
&=\frac{\lambda}{(\lambda-1)^2}+\frac{a b\lambda^{p a b}\bigl((p+1)\lambda^{a b}-p\bigr)}{(\lambda^a-1)(\lambda^b-1)}\\ 
&\quad -\frac{\lambda^{p a b}(\lambda^{a b}-1)\bigl((a+b)\lambda^{a+b}-a\lambda^a-b\lambda^b\bigr)}{(\lambda^a-1)^2(\lambda^b-1)^2}\,. 
\end{align*}
\label{cor-w}
\end{Cor} 

\noindent 
{\it Remark.}  
When $p=0$ in Corollary \ref{cor-w}, it is reduced to \cite[Theorem 1]{KZ0}. 
\bigskip 

%%%%%%%%%%%%%%%%%
%%%%%%%%%%%%%%%%%

If $\mu=1$, $\lambda\ne 1$ and $\lambda^{a_1}=\lambda^{a_2}=\cdots=\lambda^{a_k}=1$, then $\gcd(a_1,a_2,\dots,a_k)\ne 1$. So, we can choose $a_j$ such that $\lambda^{a_j}\ne 1$. 
Nevertheless, if $\lambda\ne 1$ and $\lambda^{a_1}=1$, then we have the following.   

\begin{theorem} 
If $\lambda\ne 0,1$ and $\lambda^{a_1}=1$, then for $p\ge 0$ 
$$
s_{\lambda,p}(a_1,a_2,\dots,a_k)=\frac{1}{2 a_1}\sum_{i=0}^{a_1-1}\bigl(m_i^{(p)}\bigr)^2\lambda^i-\frac{1}{2}\sum_{i=0}^{a_1-1}m_i^{(p)}\lambda^i+\frac{\lambda}{(\lambda-1)^2}\,.
$$ 
\label{th1b}
\end{theorem}  
\begin{proof} 
Since $\lambda^{a_1}=1$, the weighted sum of elements in $\mathbb N_0\backslash S_p(a_1,a_2,\dots,a_k)$ congruent to $i$ ($1\le i\le a_1-1$) modulo $a_1$ is given by 
\begin{align*}
&\sum_{j=1}^{\ell_i}\lambda^{m_i^{(p)}-j a_1}(m_i^{(p)}-j a_1)=m_i^{(p)}\sum_{j=1}^{\ell_i}\lambda^i-a_1\sum_{j=1}^{\ell_i}\lambda^i j\\
&=m_i^{(p)}\ell_i\lambda^i-a_1\frac{\ell_i(\ell_i+1)}{2}\lambda^i\\
&=\frac{m_i^{(p)}(m_i^{(p)}-i)\lambda^i}{a_1}-\frac{(m_i^{(p)}-i)^2\lambda^i}{2 a_1}-\frac{(m_i^{(p)}-i)\lambda^i}{2}\,. 
\end{align*} 
For $i=0$ and $p\ge 1$, it is given by 
\begin{align*}
&\sum_{j=1}^{\ell_i}\lambda^{m_0^{(p)}-j a_1}(m_0^{(p)}-j a_1)\\
&=\frac{m_0^{(p)}(m_0^{(p)}-a_1)}{a_1}-\frac{(m_0^{(p)}-a_1)^2}{2 a_1}-\frac{(m_0^{(p)}-a_1)}{2}\\
&=\frac{\bigl(m_0^{(p)}\bigr)^2}{2 a_1}-\frac{m_0^{(p)}}{2}\,. 
\end{align*} 
This is also valid for $i=0$ and $p=0$.  
Therefore, 
\begin{align*}  
&s_{\lambda,p}(a_1,a_2,\dots,a_k)\\
&=\frac{1}{a_1}\left(\sum_{i=1}^{a_1-1}\bigl(m_i^{(p)}\bigr)^2\lambda^i-\sum_{i=1}^{a_1-1}i m_i^{(p)}\lambda^i\right)\\
&\quad -\frac{1}{2 a_1}\left(\sum_{i=1}^{a_1-1}\bigl(m_i^{(p)}\bigr)^2\lambda^i-2\sum_{i=1}^{a_1-1}i m_i^{(p)}\lambda^i+\sum_{i=1}^{a_1-1}i^2\lambda^i\right)\\
&\quad -\frac{1}{2}\left(\sum_{i=1}^{a_1-1}m_i^{(p)}\lambda^i-\sum_{i=1}^{a_1-1}i\lambda^i\right)+\frac{\bigl(m_0^{(p)}\bigr)^2}{2 a_1}-\frac{m_0^{(p)}}{2}\\
&=\frac{1}{2 a_1}\sum_{i=1}^{a_1-1}\bigl(m_i^{(p)}\bigr)^2\lambda^i-\frac{1}{2}\sum_{i=1}^{a_1-1}m_i^{(p)}\lambda^i\\
&\quad -\frac{1}{2 a_1}\frac{a_1^2(\lambda-1)-2 a_1\lambda}{(\lambda-1)^2}+\frac{1}{2}\frac{a_1}{\lambda-1}+\frac{\bigl(m_0^{(p)}\bigr)^2}{2 a_1}-\frac{m_0^{(p)}}{2}\\ 
&=\frac{1}{2 a_1}\sum_{i=0}^{a_1-1}\bigl(m_i^{(p)}\bigr)^2\lambda^i-\frac{1}{2}\sum_{i=0}^{a_1-1}\bigl(m_i^{(p)}\bigr)\lambda^i+\frac{\lambda}{(\lambda-1)^2}\,.
\end{align*}
\end{proof}

\bigskip 

In particular, when $\lambda=-1$ and $a_1$ is odd in Theorem \ref{th1}, 
we have the formula for alternate sums. When $k=2$ and $A=\{a,b\}$, the formulas are obtained in terms of Bernoulli or Euler numbers in \cite{wang08}. 

\begin{Cor}  
When $a_1$ is odd, we have 
$$
s_{-1,p}(a_1,a_2,\dots,a_k)=-\frac{1}{2}\sum_{i=0}^{a_1-1}(-1)^{m_i^{(p)}}m_i^{(p)}+\frac{a_1}{4}\sum_{i=0}^{a_1-1}(-1)^{m_i^{(p)}}-\frac{1}{4}\,.
$$ 
\label{cor1}
\end{Cor}

When $k=2$ in Theorem \ref{th1}, by (\ref{eq:p=2}), 
if $\lambda^b\ne 1$, we have the formula in Corollary \ref{cor-w}.  
If $\lambda\ne 1$ and $\lambda^b=1$, we have 
\begin{align}  
&s_{\lambda,p}(a,b)\notag\\
&=\frac{1}{\lambda^{a}-1}\sum_{i=0}^{a-1} b(p a+i)\lambda^{b(p a+i)}
-\frac{a\lambda^{a}}{(\lambda^{a}-1)^2}\sum_{i=0}^{a-1}\lambda^{b(p a+i)}+\frac{\lambda}{(\lambda-1)^2}\notag\\ 
&=\frac{1}{\lambda^{a}-1}\sum_{i=0}^{a-1} b(p a+i)
-\frac{a^2\lambda^{a}}{(\lambda^{a}-1)^2}+\frac{\lambda}{(\lambda-1)^2}\notag\\ 
&=\frac{a b\bigl((2 p+1)a-1\bigr)}{2(\lambda^{a}-1)}-\frac{a^2\lambda^{a}}{(\lambda^{a}-1)^2}+\frac{\lambda}{(\lambda-1)^2}\,. 
\label{eq:1b1} 
\end{align}

\subsection{Proof of Theorem \ref{th-hh}} 

Similarly to the proof of Theorem \ref{th-mp}, we can choose $\ell_i=(m_i^{(p)}-i)/a_1$ for $1\le i\le a_1-1$ and $\ell_0=(m_0^{(p)}-a_1)/a_1$ for $i=0$. For simplicity, put $a=a_1$ again.  
Let $\lambda\ne 0$ and $\lambda^{a}\ne1$. 
Then, the weighted sum of elements in $\mathbb N_0\backslash S_p(a_1,a_2,\dots,a_k)$ congruent to $i$ modulo $a_1$ is given by 
$$
\sum_{j=1}^{\ell_i}\lambda^{m_i^{(p)}-j a}(m_i^{(p)}-j a)^\mu\,. 
$$ 
Notice that for $n\ge 1$, 
\begin{align*} 
\sum_{j=1}^\infty\lambda^{m_i^{(p)}-j a}j^n&=\frac{\lambda^{m_i^{(p)}}}{(1-\lambda^{-a})^{n+1}}\sum_{h=0}^{n-1}\eul{n}{h}\lambda^{-(h+1)a}\\
&=\frac{\lambda^{m_i^{(p)}}}{(\lambda^a-1)^{n+1}}\sum_{h=0}^{n-1}\eul{n}{h}\lambda^{(n-h)a}\\
&=\frac{\lambda^{m_i^{(p)}}}{(\lambda^a-1)^{n+1}}\sum_{h=1}^{n}\eul{n}{n-h}\lambda^{h a}\,.  
\end{align*}
So, for $n\ge 0$, we get 
$$
\sum_{j=1}^\infty\lambda^{m_i^{(p)}-j a}j^n=\frac{\lambda^{m_i^{(p)}}}{(\lambda^a-1)^{n+1}}\sum_{h=0}^{n}\eul{n}{n-h}\lambda^{h a}\,. 
$$ 
Hence, for $\ell_i=(m_i^{(p)}-i)/a$ ($1\le i\le a-1$) and $\ell_0=(m_0^{(p)}-a)/a$,
\begin{multline}
s^{(\mu)}_{\lambda,p}(a_1,\dots,a_k)\\
=\sum_{i=1}^{a-1}\sum_{j=1}^{\ell_i}\lambda^{m_i^{(p)}-j a}(m_i^{(p)}-j a)^\mu+\sum_{j=1}^{\ell_0}\lambda^{m_0^{(p)}-j a}(m_0^{(p)}-j a)^\mu\,. 
\label{eq:sml} 
\end{multline}  
Concerning the first term for $1\le i\le a-1$, we have 
\begin{align*}
&\sum_{i=1}^{a-1}\sum_{j=1}^{\ell_i}\lambda^{m_i^{(p)}-j a}(m_i^{(p)}-j a)^\mu\\
&=\sum_{i=1}^{a-1}\sum_{j=1}^{\ell_i}\lambda^{m_i^{(p)}-j a}\sum_{n=0}^\mu\binom{\mu}{n}\bigl(m_i^{(p)}\bigr)^{\mu-n}j^n(-a)^n\\ 
&=\sum_{i=1}^{a-1}\sum_{n=0}^\mu\binom{\mu}{n}\bigl(m_i^{(p)}\bigr)^{\mu-n}(-a)^n\sum_{j=1}^\infty\lambda^{m_i^{(p)}-j a}j^n\\
&\quad -\sum_{i=1}^{a-1}\sum_{n=0}^\mu\binom{\mu}{n}\bigl(m_i^{(p)}\bigr)^{\mu-n}(-a)^n\sum_{j=\ell_i+1}^\infty\lambda^{m_i^{(p)}-j a}j^n\\
&=\sum_{i=1}^{a-1}\sum_{n=1}^\mu\binom{\mu}{n}\bigl(m_i^{(p)}\bigr)^{\mu-n}(-a)^n\sum_{j=1}^\infty\lambda^{m_i^{(p)}-j a}j^n\\
&\quad +\sum_{i=1}^{a-1}\bigl(m_i^{(p)}\bigr)^\mu\sum_{j=1}^\infty\lambda^{m_i^{(p)}-j a}\\
&\quad -\sum_{i=1}^{a-1}\sum_{n=0}^\mu\binom{\mu}{n}\bigl(m_i^{(p)}\bigr)^{\mu-n}(-a)^n\sum_{j=1}^\infty\lambda^{i-j a}\left(j+\frac{m_i^{(p)}-i}{a}\right)^n\\ 
&=\sum_{i=1}^{a-1}\sum_{n=1}^\mu\binom{\mu}{n}\bigl(m_i^{(p)}\bigr)^{\mu-n}(-a)^n\frac{\lambda^{m_i^{(p)}}}{(\lambda^a-1)^{n+1}}\sum_{h=1}^{n}\eul{n}{n-h}\lambda^{h a}\\
&\quad +\sum_{i=1}^{a-1}\bigl(m_i^{(p)}\bigr)^\mu\frac{\lambda^{m_i^{(p)}}}{\lambda^a-1}\\
&\quad -\sum_{i=1}^{a-1}\sum_{n=0}^\mu\binom{\mu}{n}\bigl(m_i^{(p)}\bigr)^{\mu-n}(-a)^n\sum_{j=1}^\infty\lambda^{i-j a}\left(j+\frac{m_i^{(p)}-i}{a}\right)^n\\ 
&=\sum_{n=0}^\mu\frac{(-a)^n}{(\lambda^a-1)^{n+1}}\binom{\mu}{n}\sum_{h=0}^n\eul{n}{n-h}\lambda^{h a}\sum_{i=1}^{a-1}\bigl(m_i^{(p)}\bigr)^{\mu-n}\lambda^{m_i^{(p)}}\\
&\quad -\sum_{i=1}^{a-1}\sum_{j=1}^\infty\lambda^{i-j a}(i-j a)^\mu\,. 
\end{align*}
Concerning the second term for $i=0$ in (\ref{eq:sml}), the sum 
$$
\sum_{j=1}^\infty\lambda^{i-j a}\left(j+\frac{m_i^{(p)}-i}{a}\right)^n
$$  
is replaced by the sum  
$$
\sum_{j=1}^\infty\lambda^{a-j a}\left(j+\frac{m_0^{(p)}-a}{a}\right)^n\,.
$$ 
Hence, we obtain 
\begin{align*}
&s^{(\mu)}_{\lambda,p}(a_1,\dots,a_k)\\
&=\sum_{n=0}^\mu\frac{(-a)^n}{(\lambda^a-1)^{n+1}}\binom{\mu}{n}\sum_{h=0}^n\eul{n}{n-h}\lambda^{h a}\sum_{i=0}^{a-1}\bigl(m_i^{(p)}\bigr)^{\mu-n}\lambda^{m_i^{(p)}}\\
&\quad -\sum_{i=1}^{a-1}\sum_{j=1}^\infty\lambda^{i-j a}(i-j a)^\mu
-\sum_{j=1}^\infty\lambda^{a-j a}(a-j a)^\mu\,. 
\end{align*}
Since the last two terms are equal to 
\begin{align*}
-\sum_{k=0}^\infty\lambda^{-k}(-k)^\mu&=\frac{(-1)^{\mu+1}}{(1-\lambda^{-1})^{\mu+1}}\sum_{m=0}^{\mu-1}\eul{\mu}{m}\lambda^{-(m+1)}\\ 
&=\frac{(-1)^{\mu+1}}{(\lambda-1)^{\mu+1}}\sum_{j=0}^\mu\eul{\mu}{\mu-j}\lambda^j\,, 
\end{align*}
we have the desired result.

\subsection{Examples}  

Consider the sequence $14, 17, 20, 23, 26, 29$.  Then, $a=14$, $d=3$, $k=6$, $q=2$ and $r=3$. By Theorem \ref{th-hh}, we have 
\begin{align*}
&s^{(2)}_{(\sqrt[3]{2})}(14,17,20,23,26,29)\\
&=21528522+31320173525\sqrt[3]{2}+659369214\sqrt[3]{4}\,,\\ 
&s^{(3)}_{7}(14,17,20,23,26,29)\\
&=126153136547718860397749189364814847897329040723302499959511892\,,\\
&s^{(4)}_{-1/2}(14,17,20,23,26,29)\\
&=-\frac{252455039549405466513}{147573952589676412928}\,,\\
&s^{(5)}_{4+3\sqrt{-1}}(14,17,20,23,26,29)\\
&=58604955584641578954030966530484875253297329000101560480\\
&\quad -69984733631939902694215153740002368436325991046609895240\sqrt{-1}\,. 
\end{align*}
In fact, the weight power sum of nonrepresentable numbers is given by 
\begin{align*}  
&\lambda^1\cdot 1^\mu + \lambda^2\cdot 2^\mu + \lambda^3\cdot 3^\mu + \lambda^4\cdot 4^\mu + \lambda^5\cdot 5^\mu+\lambda^6\cdot 6^\mu + \lambda^7\cdot 7^\mu+ \lambda^8\cdot 8^\mu  \\ 
&+ \lambda^9\cdot 9^\mu + \lambda^{10}\cdot 10^\mu 
+\lambda^{11}\cdot 11^\mu + \lambda^{12}\cdot 12^\mu + \lambda^{13}\cdot 13^\mu + \lambda^{15}\cdot 15^\mu\\ 
&+\lambda^{16}\cdot 16^\mu + \lambda^{18}\cdot 18^\mu + \lambda^{19}\cdot 19^\mu + \lambda^{21}\cdot 21^\mu+\lambda^{22}\cdot 22^\mu + \lambda^{24}\cdot 24^\mu \\ 
&+ \lambda^{25}\cdot 25^\mu + \lambda^{27}\cdot 27^\mu 
+\lambda^{30}\cdot 30^\mu + \lambda^{32}\cdot 32^\mu + \lambda^{33}\cdot 33^\mu + \lambda^{35}\cdot 35^\mu\\ 
&+\lambda^{36}\cdot 36^\mu + \lambda^{38}\cdot 38^\mu + \lambda^{39}\cdot 39^\mu + \lambda^{41}\cdot 41^\mu+\lambda^{44}\cdot 44^\mu + \lambda^{47}\cdot 47^\mu \\ 
&+ \lambda^{50}\cdot 50^\mu + \lambda^{53}\cdot 53^\mu +\lambda^{61}\cdot 61^\mu + \lambda^{64}\cdot 64^\mu + \lambda^{67}\cdot 67^\mu\,. 
\end{align*}

%%%%%%%%%%%%%%%%%
%%%%%%%%%%%%%%%%%
%%%%%%%%%%%%%%%%%%%

For example, for the primitive $5$th root of unity $\zeta_5$, from (\ref{eq:1b1}) we have  
\begin{align*}  
s_{\zeta_5,0}(7,5)&=34\zeta_5+2\zeta_5^2+65\zeta_5^3+13\zeta_5^4\,,\\ 
s_{\zeta_5,1}(7,5)&=286\zeta_5+156\zeta_5^2+366\zeta_5^3+216\zeta_5^4+105\,,\\ 
s_{\zeta_5,2}(7,5)&=784\zeta_5+555\zeta_5^2+912\zeta_5^3+664\zeta_5^4+455\,,\\ 
s_{\zeta_5,3}(7,5)&=1525\zeta_5+1199\zeta_5^2+1703\zeta_5^3+1357\zeta_5^4+1050\,,\\ 
s_{\zeta_5,4}(7,5)&=2512\zeta_5+2088\zeta_5^2+2739\zeta_5^3+2295\zeta_5^4+1890\,,\\ 
s_{\zeta_5,5}(7,5)&=3744\zeta_5+3222\zeta_5^2+4020\zeta_5^3+3478\zeta_5^4+2975\,.
\end{align*}

\section*{Conclusion}  

In this paper, we have given convenient formulas for sums of powers and weighted sums of non-negative integers that can be expressed in at most $p$ ways in non-negative integer linear combinations of given positive integers $a_1,\dots,a_k$. As special cases of these formulas, the formulas for the $p$-genus, the $p$-Sylvester sum and so on are derived. These special case formulas are not only natural extensions of the famous classical formulas when $p=0$, but also have the great advantage that they can be applied immediately once the elements of the $p$-Ap\'ery set are known. In fact, by applying the formulas given in this paper, we have just been successful in giving explicit formulas for $p$-Frobenius numbers and $p$-Sylvester numbers for the triples of triangular numbers \cite{Ko23a}, repunits \cite{Ko23b}, Fibonacci \cite{KY}, Jacobsthal \cite{Ko23b}, and Pell numbers \cite{KM}. It is expected to be successful in other cases as well. 

This fact also shows that our $p$-generalization is not only a natural extension of existing classical results, in contrast to different generalizations of such as the Frobenius number, but also has the great merit to obtain our $p$-Frobenius and related numbers in a natural and convenient way, by using a generalization of the Ap\'ery set.

%\section*{Conflict of interest}  
 
%The author declares to have no conflict of interests. 

\section*{Data Availability Statement} 
Not applicable. 

\section*{Acknowledgement}  

The author thanks the anonymous referees for their valuable and constructive comments, which helped to improve the manuscript.

\end{document}